\documentclass[12pt]{amsart}
\usepackage{amscd,amssymb,hyperref}
\usepackage{amsfonts}
\usepackage[all]{xy}
\usepackage{graphicx, epsfig}
\newtheorem{theorem}{Theorem}[section]
\newtheorem{cor}[theorem]{Corollary}

\newtheorem{prop}[theorem]{Proposition}

\theoremstyle{definition}

\theoremstyle{remark}

\numberwithin{equation}{subsection}
\theoremstyle{plain}

\newtheorem{conjecture}{Conjecture}
\newtheorem{question}{Question}

\newtheorem{problem}{Problem}

\def\square{\vbox{
      \hrule height 0.4pt
      \hbox{\vrule width 0.4pt height 5.5pt \kern 5.5pt \vrule width 0.4pt}
      \hrule height 0.4pt}}

\def\ch\mathrm{c h}

\long\def\symbolfootnote[#1]#2{\begingroup%
\def\thefootnote{\fnsymbol{footnote}}\footnote[#1]{#2}\endgroup}

\numberwithin{equation}{section}

\begin{document}

\bigskip

\title{Spatial graph as connected sum of a planar graph and a braid}
\author{Valeriy ~G.~Bardakov}
\address{Sobolev Institute of Mathematics, Novosibirsk State University, Novosibirsk 630090, Russia}
\address{Novosibirsk State Agrarian University, Dobrolyubova street, 160, Novosibirsk, 630039, Russia}
\address{Tomsk State University, pr. Lenina, 36, Tomsk, 634050, Russia.}
\email{bardakov@math.nsc.ru}

\author{Akio Kawauchi}
\address{Department of Mathematics, Osaka City University
Sugimoto, Sumiyoshi-ku, Osaka 558-8585, Japan}
\email{kawauchi@sci.osaka-cu.ac.jp}


\thanks{The authors gratefully acknowledge the support of the grant RNF-16-11-10073. Also they thank Vera Gorbunova, who drew pictures for the article.}

\date{\today}


\begin{abstract}
In this paper we show that every finite spatial graph is a connected sum of a planar graph, which is a forest, i.e. disjoint union of finite number of trees  and a tangle.
 As a consequence we get that any finite  spatial graph is a connected sum of a planar graph and  a braid. Using these decompositions it is not difficult to find a set of generators and defining relations for the fundamental group of compliment of a spatial graph in 3-space $\mathbb{R}^3$.
\end{abstract}
\maketitle

\section{Introduction}


For studying classical links in three dimensional space $\mathbb{R}^3$ people use some  presentations of links. For example, link diagrams on a plane,  rectangular
diagram of a link  \cite{C}-\cite{C1}, link as the  closure of a braid, link as a plat \cite{Bir},  and so on.

Theory of spatial graphs  is a generalization of  link theory  in three dimensional space.
L.~Kauffman \cite{Kauf} defined two types of equivalence relations on the set of spatial graphs and prove some analogs of Reidemeister theorem. Also, he associated a collection of knots and links to a spatial graph that gives  computable invariants for spatial graphs. Other invariants were constructed in \cite{K1} (see also \cite{K}).

K.~Kanno, K.~Taniyama \cite{KT} for every oriented spatial graph  found a braid presentation. This
result is a generalization of Alexander's theorem.

In the present paper we suggest some other presentations  of a spatial graph. At first we  prove that every finite spatial graph is 
a connected sum of a planar graph, which is a forest, i.e. disjoint union of finite number of trees  and a tangle. Using the Alexander theorem,
we prove that this graph is a connected some of a forest and a braid (braid decomposition of  the spatial graph). Also, we construct another type decomposition without using the 
Alexander theorem. We prove that any spatial graph is a connected sum of a forest and a plat-braid that is some  symbiosis of a plat and a braid (plat decomposition  of  the spatial graph).

Using these decompositions it is not difficult to find a set of generators and defining relations for the fundamental group of compliment of a spatial graph in 3-space $\mathbb{R}^3$  (see Section 4).

\section{Combinatorial graphs and spatial graphs}

A combinatorial graph $\mathcal{G}$ consists of a four $(V, E, i, t)$, where $V$ is the set of {\it vertices}, $E$ is  the set of
{\it edges}, $i, t : E \to V$ are two functions, $i(e)$ is called the {\it beginning of} $e$ and $t(e)$ is called the {\it end of} $e$. We write $V(\mathcal{G})$ for the vertices of $\mathcal{G}$ and $E(\mathcal{G})$ for the edges of $\mathcal{G}$ when necessary. If $i(e) = t(e)$, then $e$ is called a {\it loop}. A graph $\mathcal{G}$ is called finite if $V$ and $E$ are finite.
A combinatorial graph can contains  multiple edges, but we will assume that $\mathcal{G}$ has no  vertices of degrees 0 and 1.

A {\it spatial graph} $\Gamma$ is a geometric realization of a combinatorial graph $\mathcal{G}$ in $\mathbb{R}^3$, that is an injective  map $\Gamma \to \mathbb{R}^3$ under which $V$ goes to a set of distinct points and any $e$ in $E$ goes to a topological interval $L_e$ that is a topological space homeomorphic to the close interval $[0, 1]$ in the set of real numbers $\mathbb{R}$, which is beginning in the image of the vertex $i(e)$ and ending in the image of the vertex $t(e)$ if $i(e) \not = t(e)$; if $i(e) = t(e)$, then $L_e$ homeomorphic to the circle. The topological interval $L_e$ can meet with the images of $V$ only in the beginning or ending points and two different topological intervals do not intersect at internal points. For examples, if $\mathcal{G}$ contains one vertex $v$ and one edge $e$, where $i(e) = t(e)$, then the set of spatial graphs is the set of knots in $\mathbb{R}^3$.
Two spatial graphs $\Gamma$ and $\Gamma'$ are {\it equivalent} if there is an orientation-preserving homeomorphism $h : \mathbb{R}^3 \longrightarrow \mathbb{R}^3$ sending $\Gamma$ onto $\Gamma'$. A fundamental topological problem (equivalence decision problem) on spatial graphs is:

\begin{problem}
By an effective method, decide whether or not two given spatial graphs of a combinatorial graph are equivalent.
\end{problem}

A diagram $D_{\Gamma}$ of a spatial graph $\Gamma$ is a regular projection of $\Gamma$ to a plane in $\mathbb{R}^3$.
Equivalence relation for spatial graphs is generated by a set
of local moves that generalize the Reidemeister moves for diagrams of knots. L.~Kauffman \cite{Kauf} added to the usual list of Reidemeister moves two moves
involving a vertex (moves IV and V in Figure 1).

\begin{center}
\includegraphics[width=10.0cm]{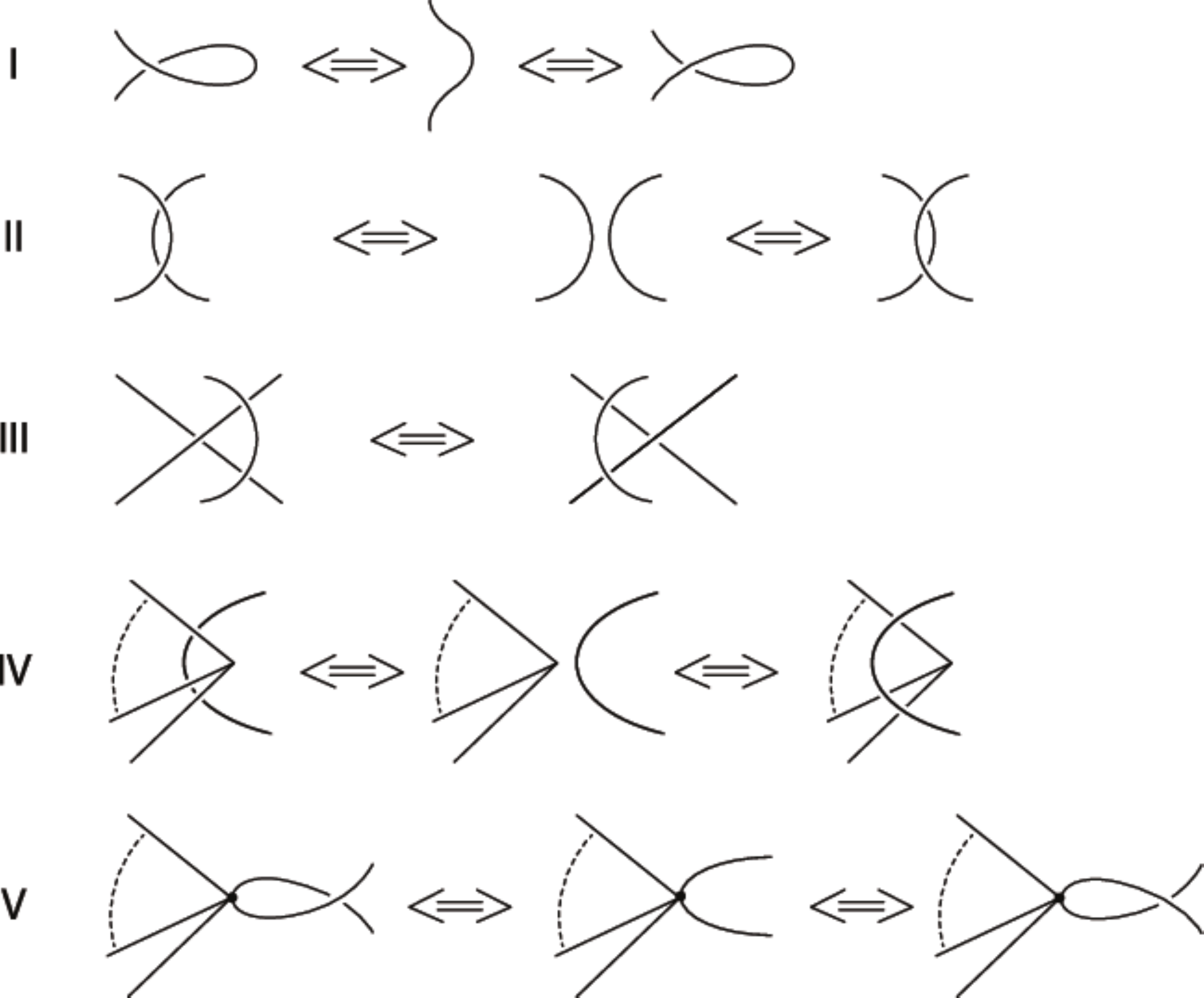}\\
{\small Fig. 1. Reidemeister moves}

\end{center}

Moves IV allows an edge to slide under or over a bundle of strands at a vertex.
Move V allows any two adjacent (in the planar diagram) strands at a vertex to twist around one another. This is the basic topological vertex move: two strands at the vertex can twist without affecting the other strands. L.~Kauffman proved that two spatial graphs $\Gamma$ and $\Gamma'$ are equivalent if and only if any diagram $D$ of $\Gamma$ is deformed into any diagram $D'$ of $\Gamma'$  by a finite sequence of the generalized Reidemeister moves I--V.

Recall that a connected non-empty graph is called a {\it tree} if it does not have cycles and multiples edges, a connected non-empty graph is called a
{\it forest} if it is a disjoint union of trees. It is evident, that every forest is a planar graph.

We say that a spatial graph $\Gamma$ is a {\it connected sum} of two spatial subgraphs $\Gamma_1$ and $\Gamma_2$ and write $\Gamma = \Gamma_1 \sharp \Gamma_2$, if  there is a 3-ball $B \subset \mathbb{R}^3$ such that its boundary $S^2 = \partial B$ does not contain vertices of $\Gamma$, the intersection $S^2 \cap \Gamma$ is a finite number of points, the closure of $(\mathbb{R}^3 - B) \cap \Gamma$ is equal to $\Gamma_1$ and the closure of $B \cap \Gamma$ is equal to $\Gamma_2$.

\section{Braid decomposition}

The main result of the present section is

\begin{theorem} \label{t1}
Let $\Gamma$ be a finite spatial graph in $\mathbb{R}^3$. Then there are a forest $T_0$ and a braid $\beta$ such that $\Gamma = T_0~ \sharp ~ \beta$.
\end{theorem}

For illustration of this theorem see Fig. 2, where  a forest is in the left box and a braid in the right box.

\begin{center}
\includegraphics[width=8.0cm]{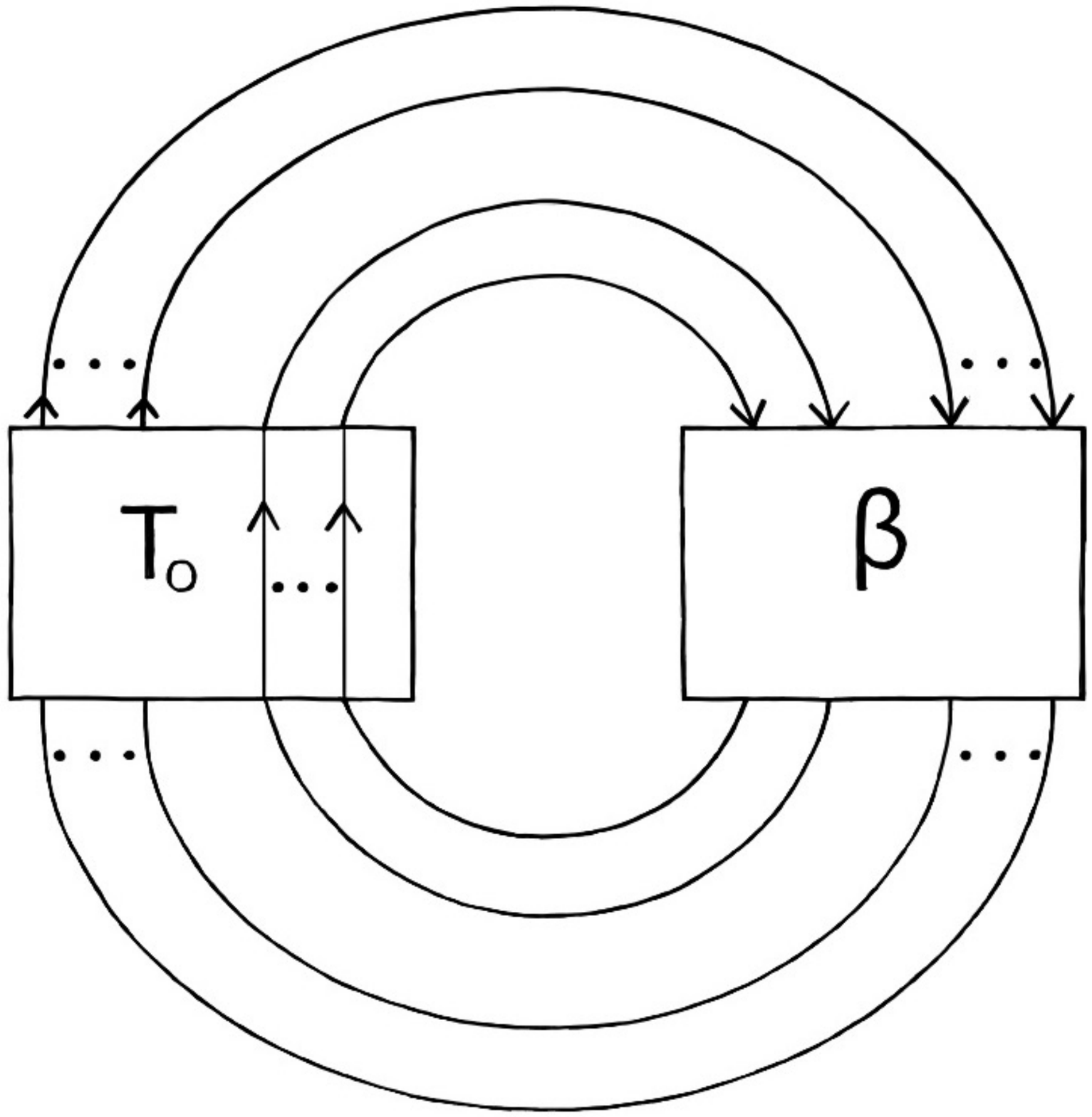}\\
{\small Fig. 2. Spatial graph as a connected sum of forest and braid}
\end{center}


We  prove the theorem for connected graphs. The general case is similar.
Let $\Gamma$ be a finite connected spatial graph in $\mathbb{R}^3$. Denote by $V(\Gamma)$ the set of all vertices of $\Gamma$ and by $E(\Gamma)$ the set
of all edges of $\Gamma$. Take some maximal tree $T$ of $\Gamma$. By the definition of maximal tree, the set of  $V(T)$  is equal to the set
$V(\Gamma)$ and the set of edges $E(T)$ is a subset of $E(\Gamma)$. If $E(T) = E(\Gamma)$, then $\Gamma = T$ and the theorem is true. Hence, we shall
assume that $E(T) \not= E(\Gamma)$. We  call the edges in $E(T)$ by {\it tree edges} and the edges in $E(\Gamma) \setminus E(T)$ by {\it spatial edges}.
The set $V(T)$ is the disjoint union of two subsets: $V(T) = I(T) \sqcup S(T)$, where $I(T)$ is the set of {\it inner vertices}, i.e. vertices  which are
incident only tree edges; $S(T)$ is the set of {\it spatial vertices}, i.e. vertices  which are
incident some  spatial edges.
Every  spatial edge $e$ has the  initial vertex $i(e)$ and the  terminal vertex $t(e)$, which lie in $S(T)$.

Transform the graph $\Gamma$ in $\mathbb{R}^3$ by the such manner that the tree $T$ lies on the plane $xOy$ and all spatial edges  lie in the
subspace $z \geq 0$ (see Fig. 3). To do it we can consider a diagram of $\Gamma$ and by using the generalized Reidemeister moves transform it to the needed diagram.

\begin{center}
\includegraphics[width=8.0cm]{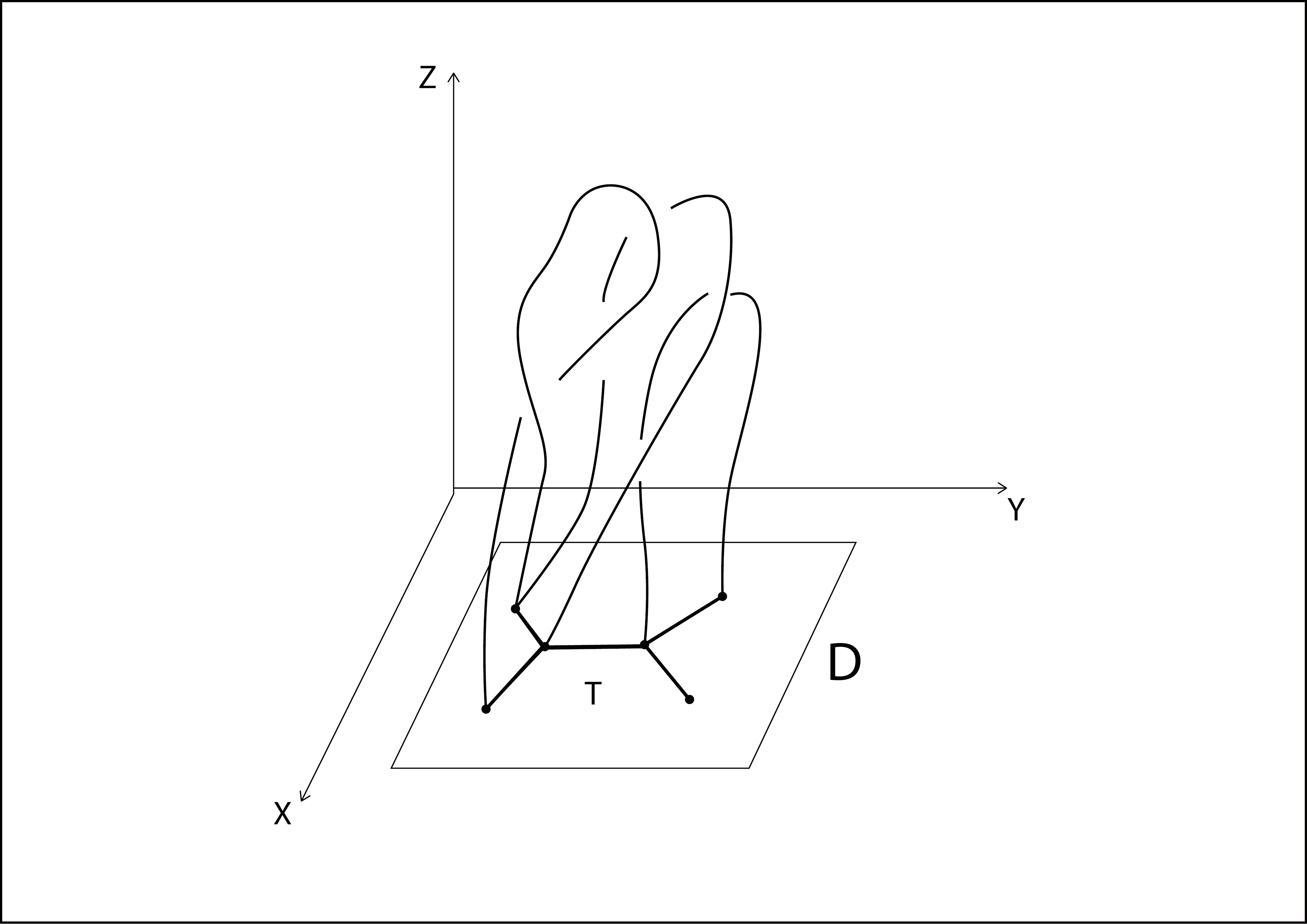}\\
{\small Fig. 3. A maximal tree $T$ on the plane}

\end{center}

We  assume that the tree $T$ lies inside of some rectangle $D$ those sides are parallel to the axes $Ox$ and $Oy$.
Take a some vertex in $S(T)$ and  call this vertex by the initial vertex  and  shift $T$ to a position where this vertex lies in the upper part of $D$ with respect to the $y$-coordinate; the vertices which are connected to the initial vertex by some edge in $E(T)$ lie in a lower level with respect to $y$-coordinate; the vertices 
which are connected to the vertices 
of the second level by some edge in $E(T)$  
lie in the third level and so on.  See the
upper picture of Fig. 3,
where vertex $a$
is the initial vertex and lies in the first level, vertices $b$, $c$ and $d$ lie in the second level and so on. Using induction by the number of levels
transform the graph $\Gamma$ to a graph in which all the vertices in $V(T)$ lie in the middle line of $D$. See the lower picture in Fig. 4, where we transformed the vertex $a$ from the first level to the second level.

\begin{center}
\includegraphics[width=8.0cm]{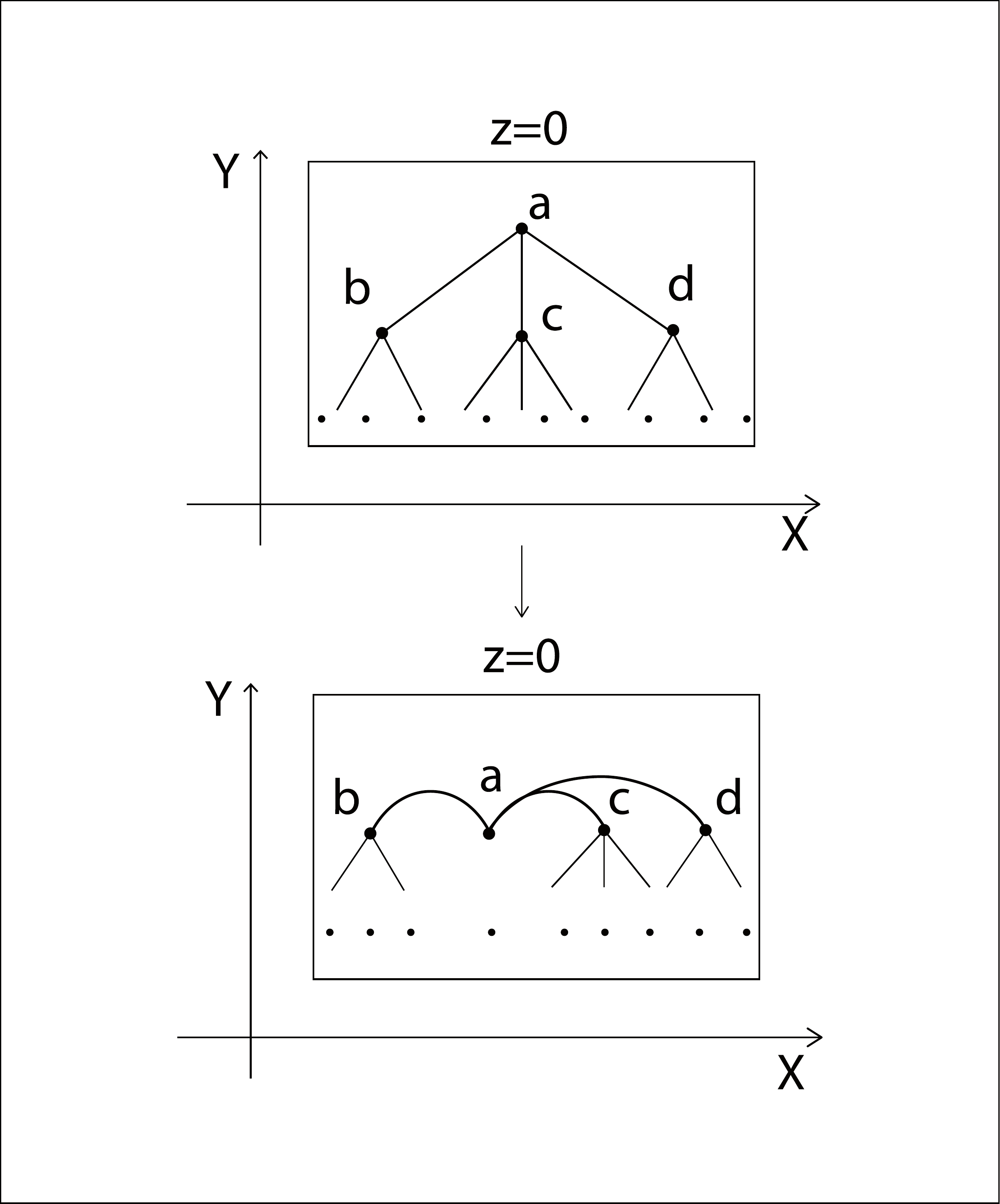}\\
{\small Fig. 4. Transformation of the tree}

\end{center}

For every vertex $v \in S(T)$ construct a square $V_v = V_v^0$ on the plane $xOy$, such that $v$ is the center of this square and the sides of this square are parallel to the axes
$Ox$ and $Oy$; $V_v$ does not contain the other vertices of $T$. Let $V_v^1$ be the orthogonal projection of $V_v$ onto the plane $z = 1$. Move
the spatial edges for which $v$ is the initial vertex to the position in which these edges intersect the upper side (side with larger $y$-coordinate)
of $V_v^1$ and the spatial edges
for which $v$ is the terminal vertex to the position in which these edges intersect the bottom side  (side with smaller $y$-coordinate) of $V_v^1$
(see Fig. 5).

\medskip

\begin{center}
\includegraphics[width=8.0cm]{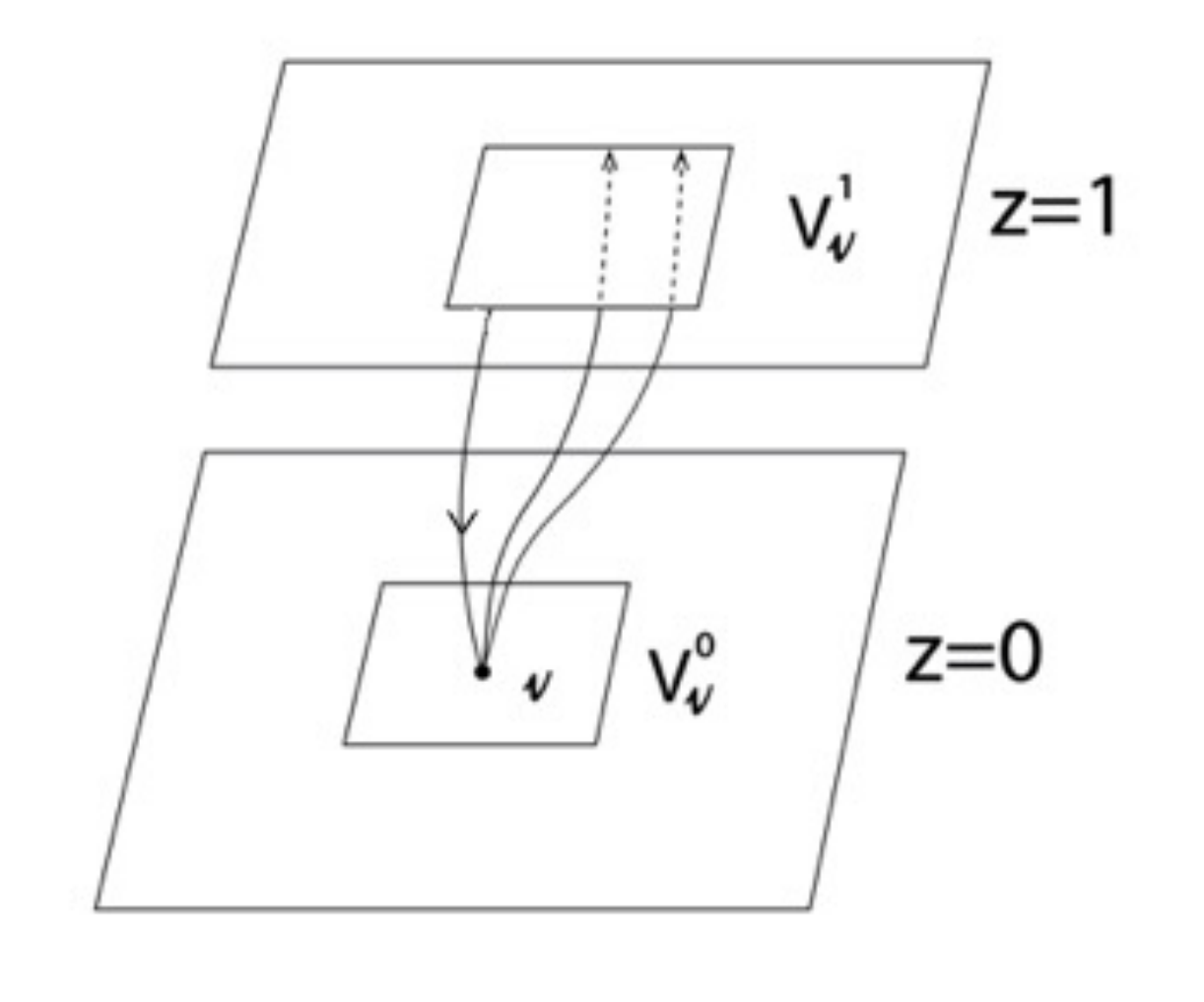}\\
{\small Fig. 5. Spatial edges in the neighborhood of a vertex}

\end{center}

\medskip

Do  this for all spatial vertices
to get in the plane $z = 1$ a picture as in Fig. 6.

Denote by $\overline{T}$ the closure of the intersection of $\Gamma$ with the subspace $z \leq 1$. We see that  $\overline{T}$  is a tree. Consider the
subspace $z \geq 1$. The intersection of this subspace with $\Gamma$ is the union of strings.

\medskip

\begin{center}
\includegraphics[width=15.0cm]{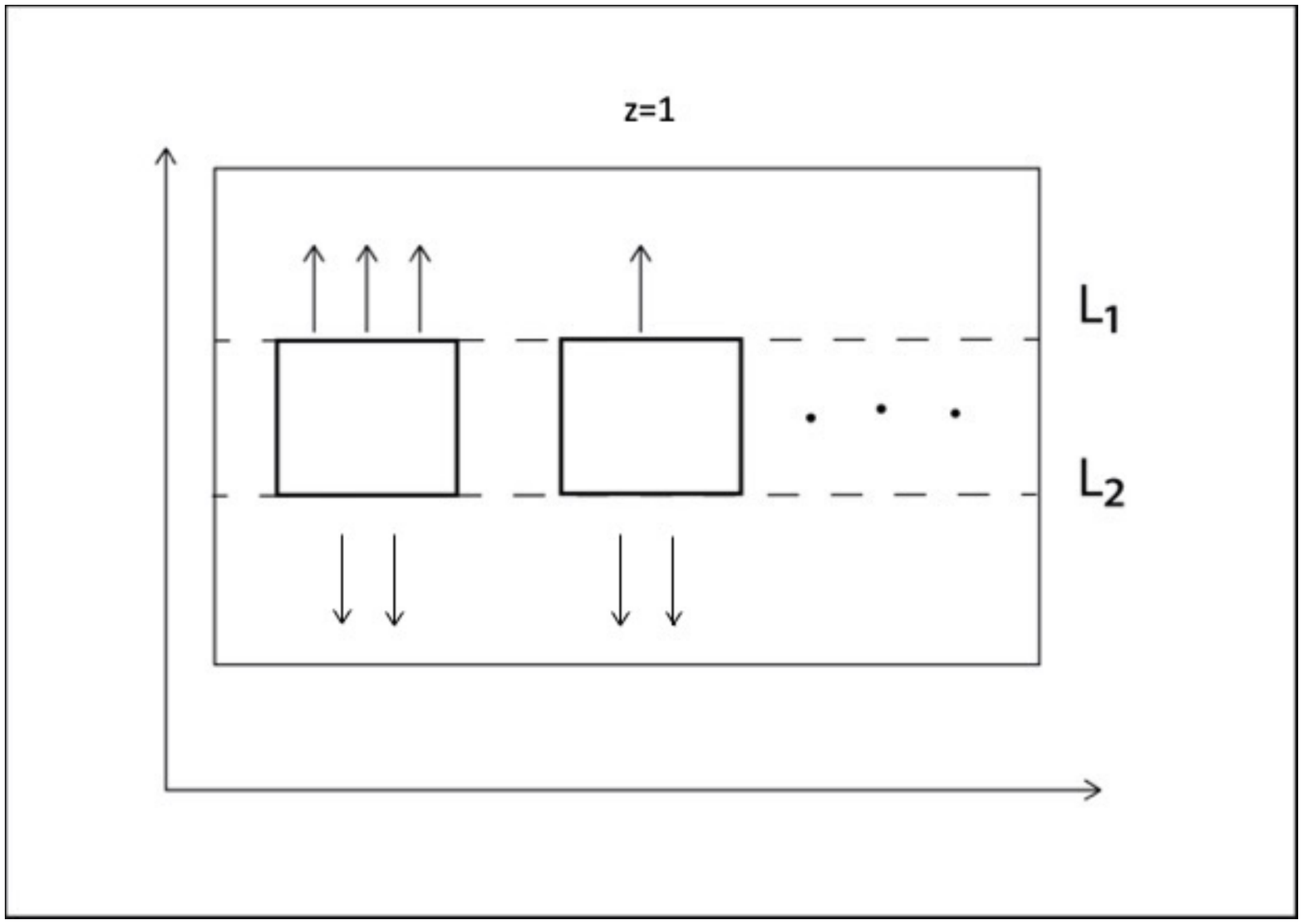}\\
{\small Fig. 6. Rectangles for the vertices }

\end{center}

\medskip

We can put these  strings in the box $D_1 \times [0,1]$, where,  $D_1$ be a rectangle in the plane $z = 1$  with sides that are parallel
to the coordinate lines $Ox$ and $Oy$. We get 
some $(m,m)$-tangle $\overline{\beta}$ (see Fig. 7), where $m$ is the set of spatial edges in $\Gamma$. Hence, $\Gamma$ is the connected sum of  $\overline{T}$  and $\overline{\beta}$. Tangle $\overline{\beta}$
does not contain free components and each his edge is stating  at the upper side and ending at the bottom side of the parallelepiped $D_1 \times [0,1]$. 

\medskip

\begin{center}
\includegraphics[width=8.0cm]{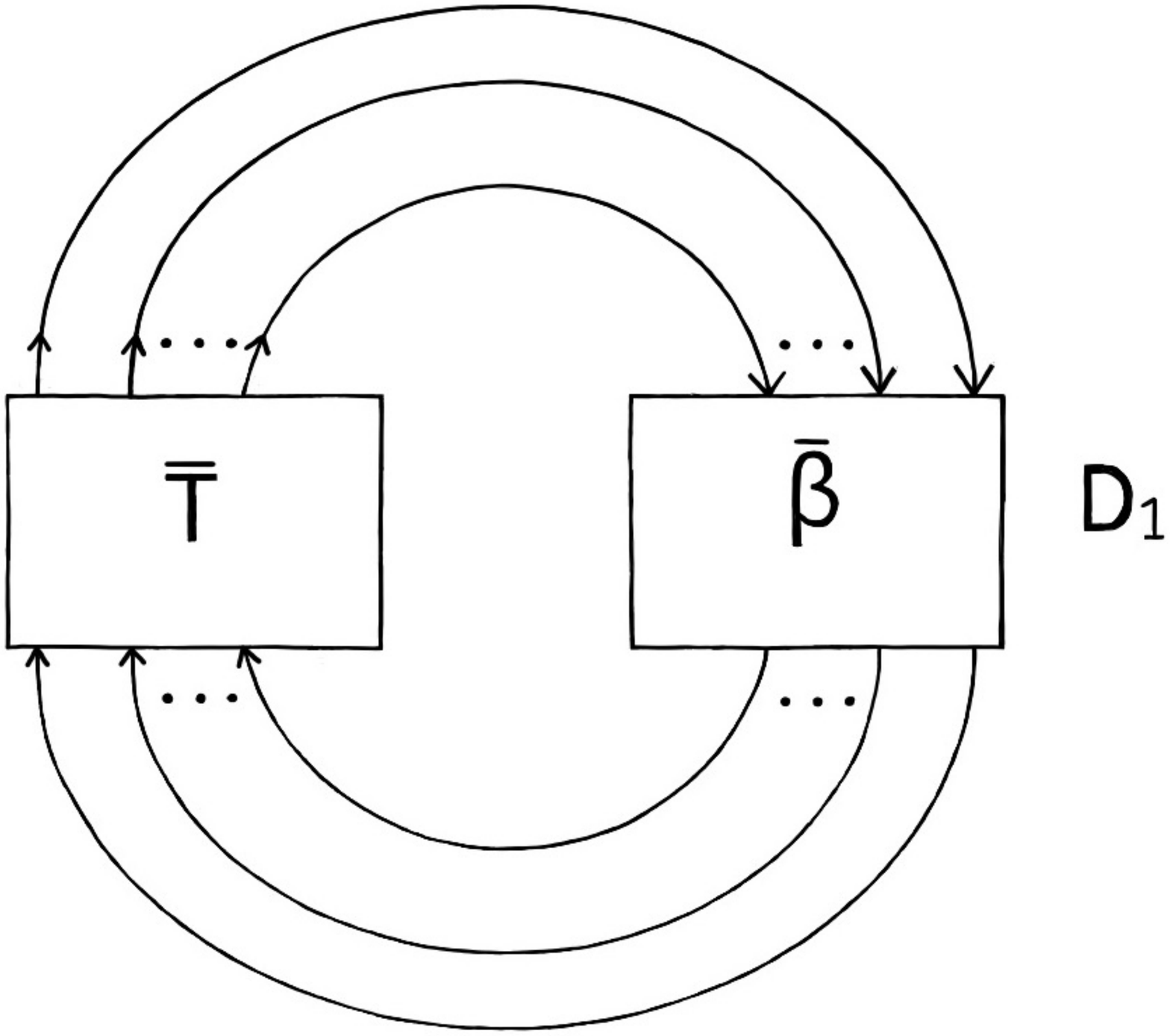}\\
{\small Fig. 7. Tree and tangle}

\end{center}

\medskip

Hence, we have proven

\begin{prop} \label{p1}
Any finite spatial graph $\Gamma$ is a connected sum of a forest $\overline{T}$ and a tangle $\overline{\beta}$.
\end{prop}

In this decomposition the forest $\overline{T}$ contains $2m$ edges, where $m = |E(\Gamma)| - |E(T)|$, which do not lie in the maximal forest $T$ of $\Gamma$. From these edges $m$ have beginning vertex of valence 1, denote them $a_1, a_2, \ldots, a_m$ and  will call {\it incoming edges}. Other $m$ edges have terminating vertex of valence 1, denote them $b_1, b_2, \ldots, b_m$ and  will call {\it outgoing edges}. The tangle $\overline{\beta}$ is $(m,m)$-tangle that is a tangle with $m$ incoming strings and $m$ outgoing strings. To get the spatial graph $\Gamma$ we must gluing the outgoing edges of the forest with the incoming strings of the tangle and the outgoing strings of the tangle with the incoming edges of  $\overline{T}$.

Using  Proposition  \ref{p1} we can prove Theorem \ref{t1}.

\begin{proof}
We assume that the spatial graph $\Gamma$ is connected and 
by Proposition \ref{p1} it is the connected sum of  the tree  $\overline{T}$  and the tangle $\overline{\beta}$. Using the same idea as in the proof of the  Alexander theorem
\cite[Chapter 2.1]{Bir}, we can present $\overline{\beta}$ as a connected sum of a braid $\beta$ and some number of unknotted and unlinked arcs.
Adding these arcs to the tree $\overline{T}$ we get the forest $T_0$ and the decomposition of $\Gamma$  as the connected sum of    $T_0$  and the braid $\beta$ (see Fig. 2).
\end{proof}

\section{Plat decomposition}

Any link can be present as a closure of a braid or as a plat (see \cite{Bir}). We introduce some object that is a symbiosis of a plat and a braid.  Let $k, m$ be non-negative integer numbers
and
$$
I^3 = \{ (x, y, z) \in \mathbb{R}^3~|~0 \leq x, y, z \leq 1 \}
$$
be the cub with the side 1. On the upper side take $2k+m$ points $P_1$,  $P_2$, $\ldots$, $P_{2k+m}$, which lie in the plane $y = 1/2$ and the point with bigger index has bigger $x$-coordinate.  On the  bottom side also 
take $2k+m$ points $Q_1$,  $Q_2$, $\ldots$, $Q_{2k+m}$, which also  lie in the plane $y = 1/2$ and the point with bigger index has bigger $x$-coordinate.

The $m$-component $k$-{\it plat-braid} or simply $(k, m)$-plat-braid $L_{k,m} = L_1 \sqcup L_2 \sqcup \ldots \sqcup L_m$ is the disjoint union of $m$ topological intervals, which are images   of $m$ segments $[0, 1]$ into the 3-space $\mathbb{R}^3$ and  the following conditions hold:
 
1) all points $P_i$ and  $Q_i$ lie on $L_{k,m}$ and 
$$
L _{k,m} \cap \partial  I^3 = \{ P_1,  P_2, \ldots, P_{2k+m}, Q_1,  Q_2, \ldots, Q_{2k+m} \};
$$

2) $\partial L_{k,m}$ lie in $\mathbb{R}^3 \setminus I^3$, the first intersection point of  $L_i$ and $I^3$ is $P_i$ and the last intersection point of  $L_i$ and $I^3$ is $Q_i$
for all $i \in \{ 1, 2, \ldots, m \}$.

3) Any pair of points $\{P_j, P_{j+1} \}$, $\{Q_j, Q_{j+1} \}$, $j = 1, 2, \ldots, k-1$ is connected by some arc which is a part of $L_{k,m}$.

4) If we forget about orientation of $L_{k,m}$, then the intersection $L _{k,m} \cap  I^3$ is an $2k+m$-braid $\gamma$.

For example of $(2,2)$-plat-braid see Figure 8.

\begin{center}
\includegraphics[width=5.0cm]{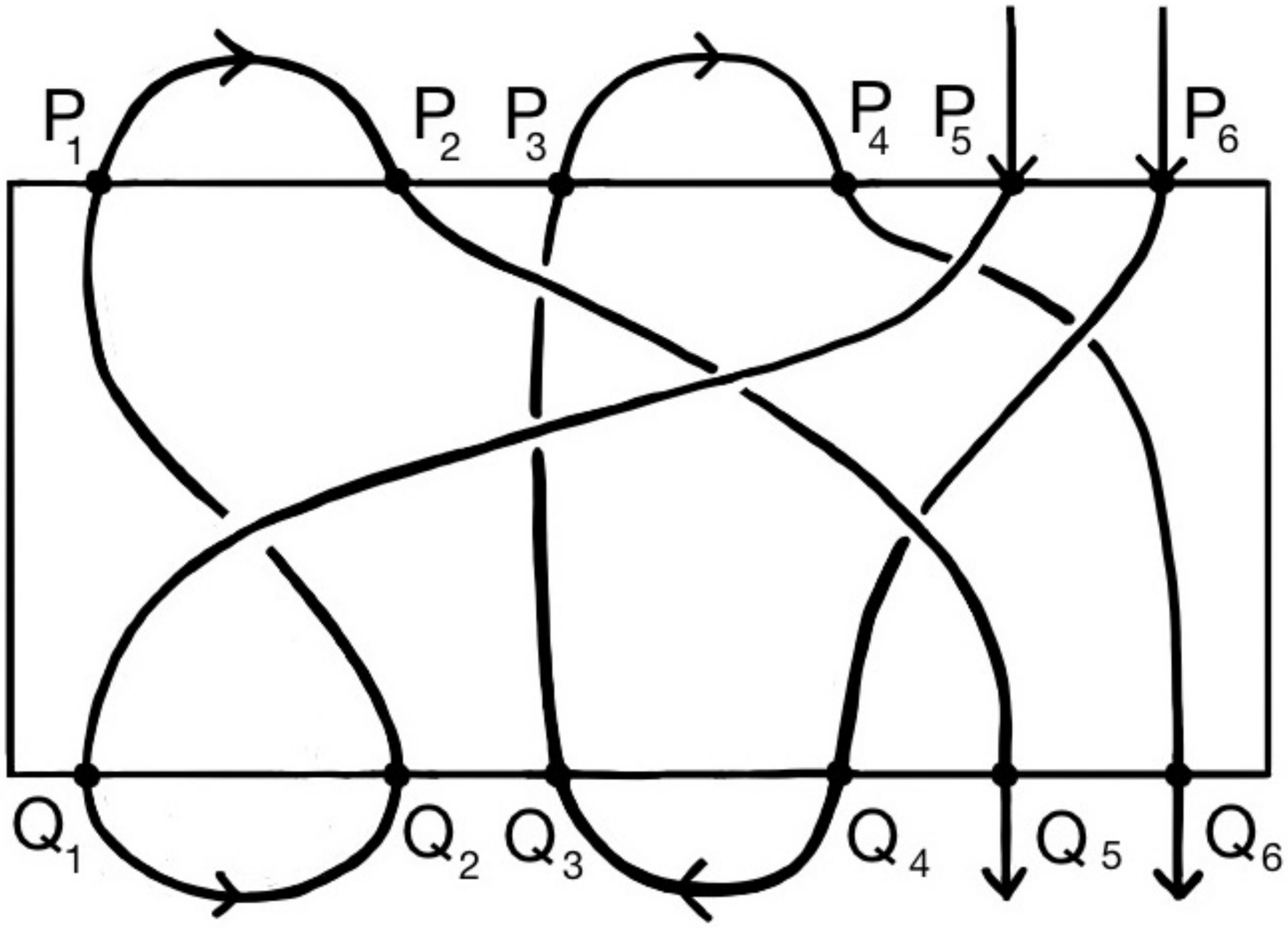}\\
{\small Fig. 8. Example of $(2,2)$-plat-braid}
\end{center}

The main result of the present section is

\begin{theorem} \label{t2}
Let $\Gamma$ be a finite spatial graph with $n$ connected components in $\mathbb{R}^3$. Then there are a forest $\overline{T}$, which is a disjoint union of $n$ trees and a plat-braid $\gamma$ such that $\Gamma = \overline{T}~ \sharp ~ \gamma$.
If $\Gamma$ is a connected spatial graph, then the forest $\overline{T}$ is a tree.
\end{theorem}

For illustration of this theorem see Figure 9. We shall call the decomposition of  $\Gamma$ from this theorem by {\it plat decomposition}

\begin{center}
\includegraphics[width=5.0cm]{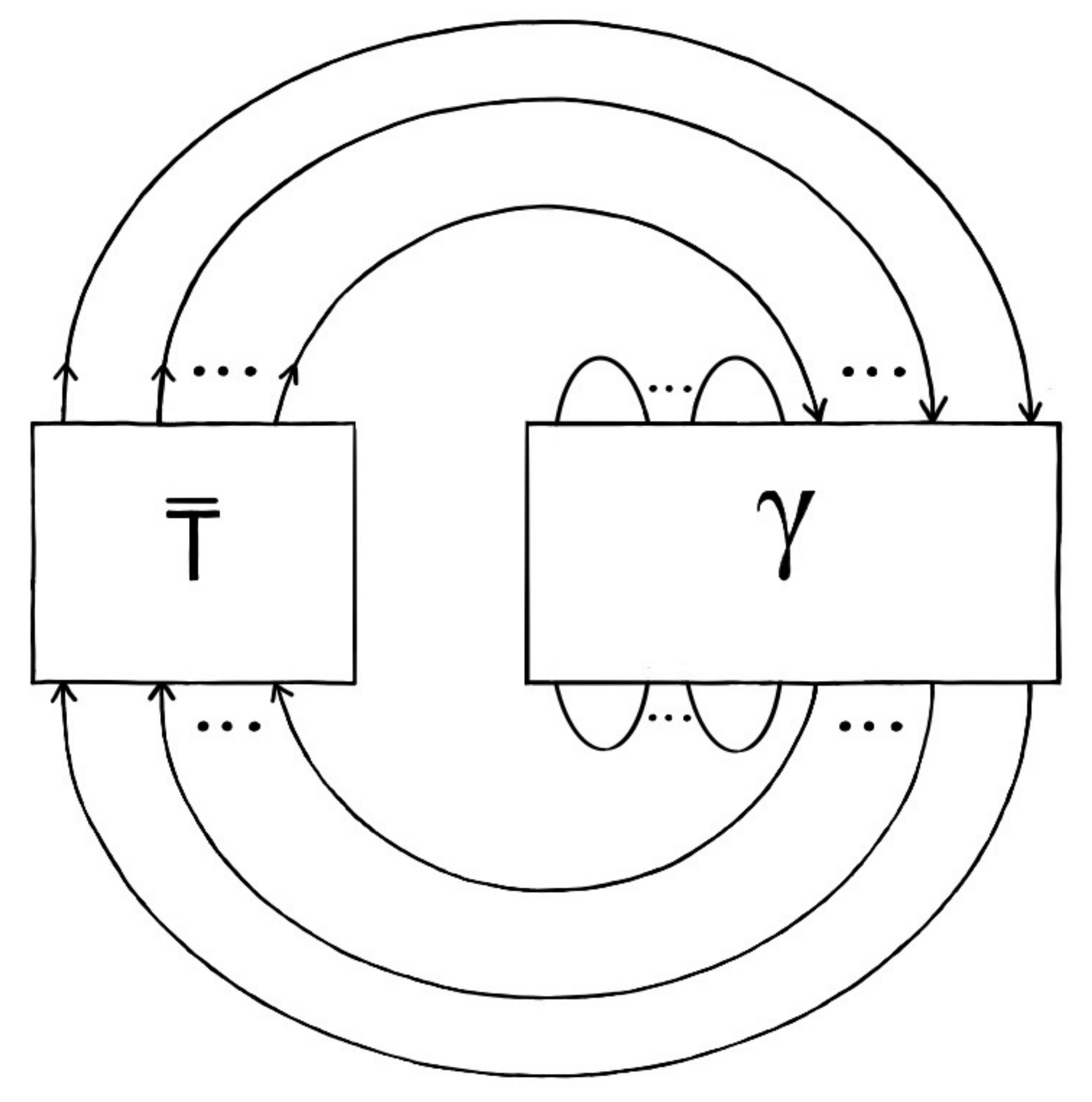}\\
{\small Fig. 9. Spatial graph as decomposition of a forest and a plat-braid}
\end{center}

\begin{proof}
We assume that the spatial graph $\Gamma$ is connected and 
by Proposition \ref{p1} it is the connected sum of  the tree  $\overline{T}$  and the tangle $\overline{\beta}$.

We shall transform $\overline{\beta}$ to get some plate-braid. We can assume that under the projection of $\overline{\beta}$ onto the plane $z = 1$ all crossings lie inside $D_1$ and  that this projection is regular projection. This projection has finite number of local maximums and local minimums with respect to
$y$-coordinate.
We can check that the number of the local maximums is equal to the number of  local minimums. Denote by $m_1$, $m_2$, $\ldots$, $m_k$ the points
of local maximums and by $l_1$, $l_2$, $\ldots$, $l_k$ the points of local minimums (see Fig. 10). Using isotopy we will move the arcs with
maximum to the upper side of $D_1$ and the arcs with minimum to the bottom side of $D_1$.

\medskip

\begin{center}
\includegraphics[width=8.0cm]{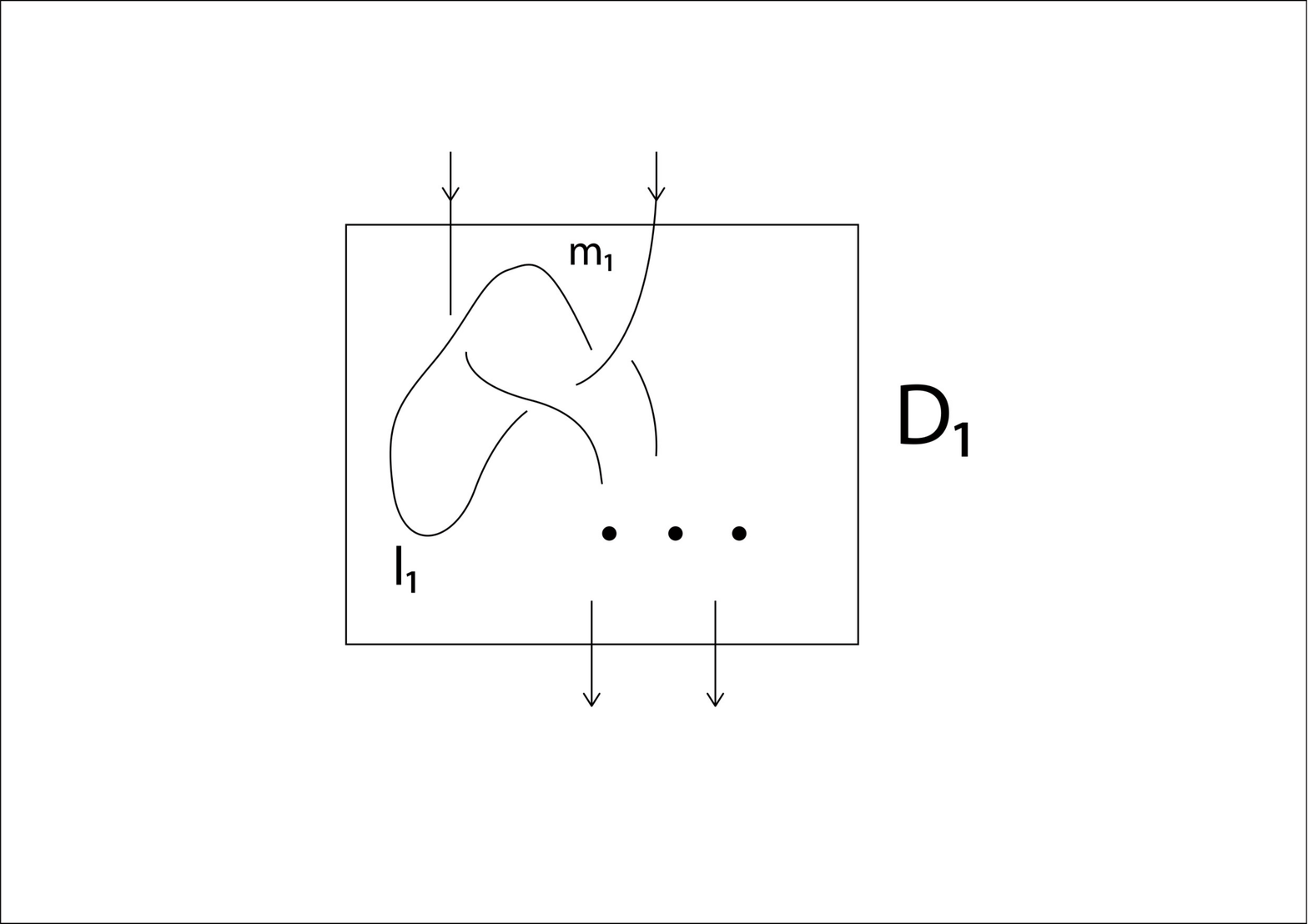}\\
{\small Fig. 10. Local maximum and local minimum}

\end{center}

\medskip


Moving this arcs until they will be outside of $D_1$. Using the transformations as in Figure 11 we can assume that outside of $D_1$ all arcs
with maximum lie on the left side from the edges which go inside the rectangle and all arcs with minimum lie on the left side from the edges which go
outside the rectangle.

\medskip

\begin{center}
\includegraphics[width=8.0cm]{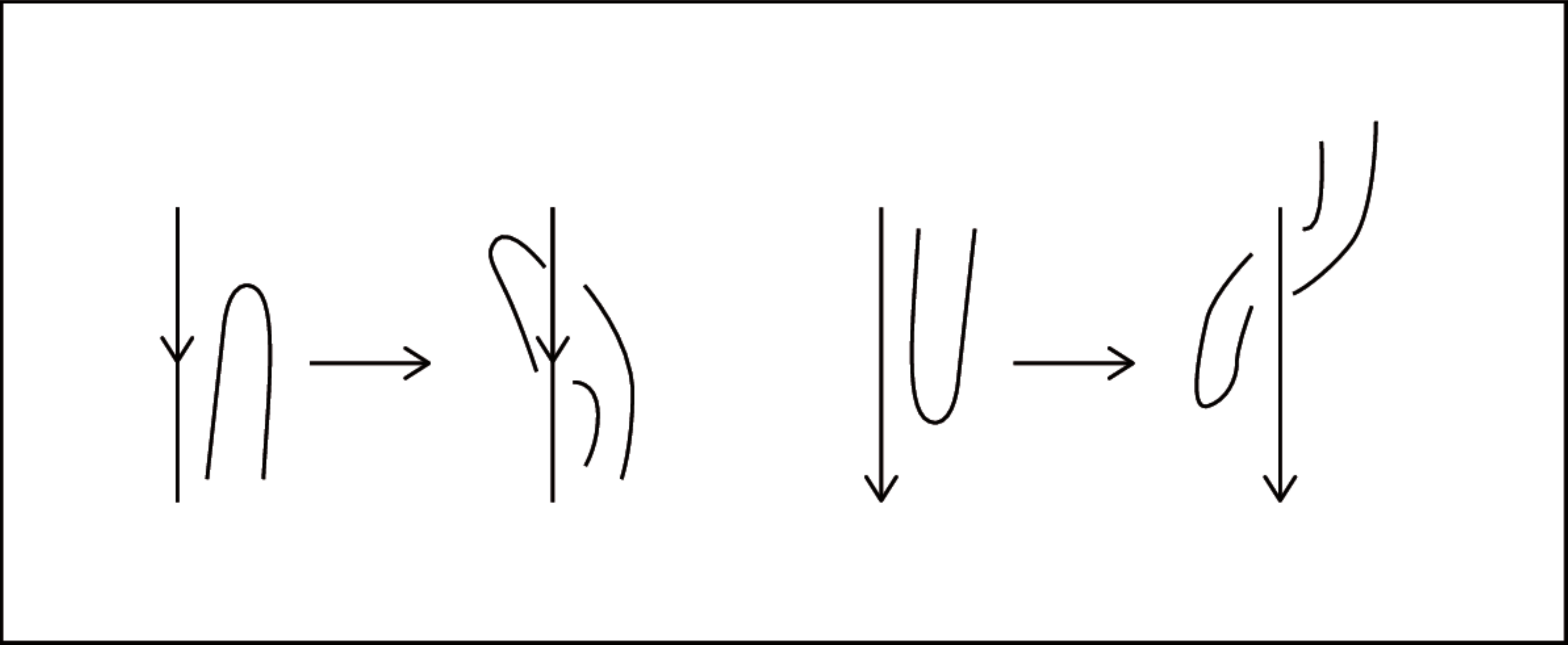}\\
{\small Fig. 11. Moving  of arcs with local maximum and minimum }

\end{center}

\medskip

At the end we  transform the $(m, m)$-tangle $\overline{\beta}$ to the $(k, m)$-plat-braid $\gamma$ in which the number of connected components is
$m = |E(\Gamma)| - |E(T)|$ (see  Fig. 12).

\medskip

\begin{center}
\includegraphics[width=8.0cm]{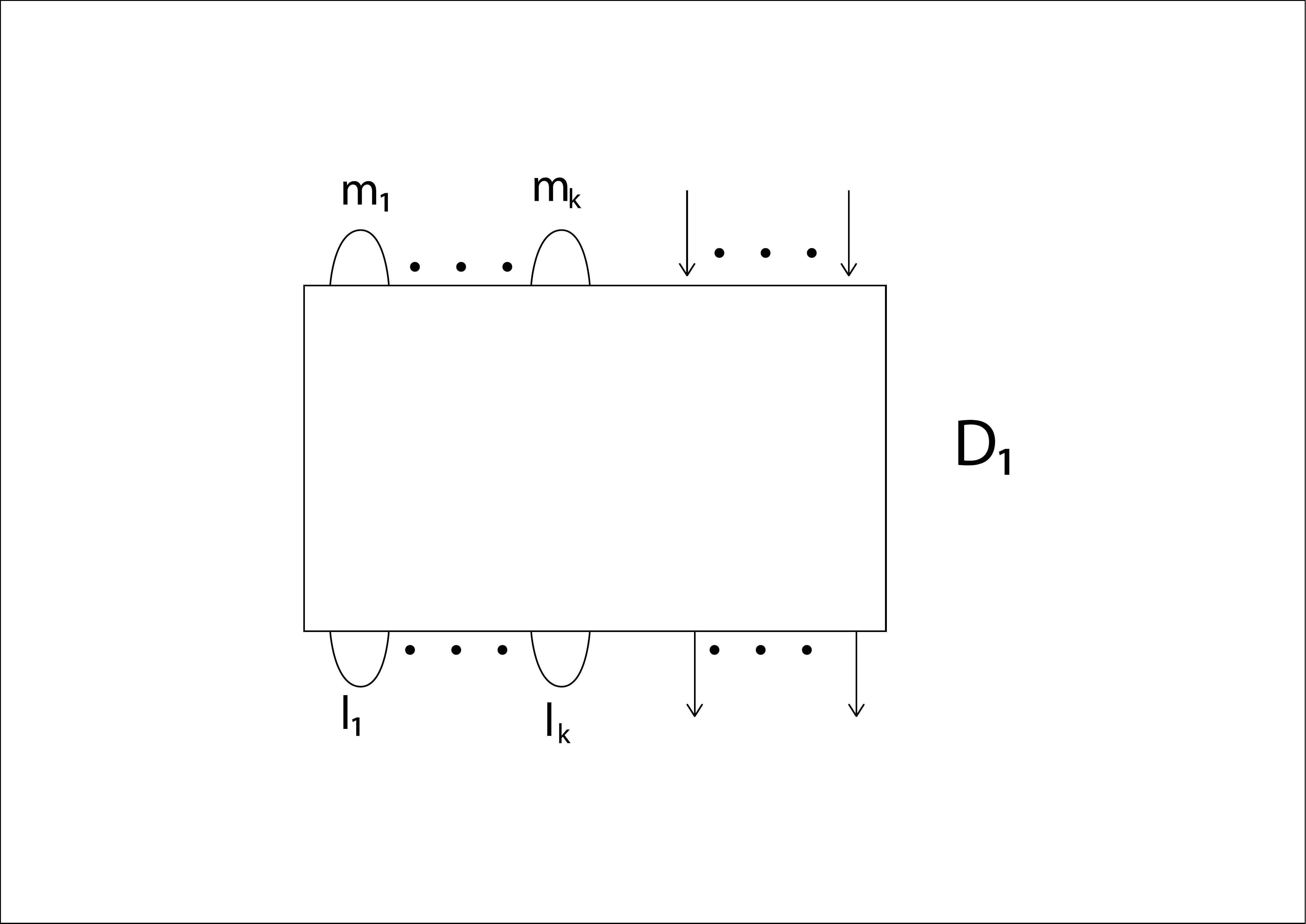}\\
{\small Fig. 12. The $(k, m)$-plat-braid $\gamma$}

\end{center}

\medskip

Take the decomposition $\Gamma = \overline{T} ~ \sharp ~\overline{\beta}$, constructed in Proposition \ref{p1}, cut the tangle $\overline{\beta}$ and paste
the $(k, m)$-plat-braid $\gamma$, then we get the plat decomposition $\Gamma = \overline{T} ~ \sharp ~\gamma$.

\end{proof}

\bigskip

\section{Some applications}

\subsection{Braid index} Using the braid decomposition we can introduce some invariant of spatial graph.  The {\it braid index $bi(T_0 ~ \sharp ~\beta)$ of the braid decomposition} $\Gamma = T_0 ~ \sharp ~\beta$ is the number of strings in the braid  $\beta$. The {\it braid index} $bi(\Gamma)$ of the spatial graph $\Gamma$ is the minimum $bi(T_0 ~ \sharp ~\beta)$ by all possible  braid decompositions $\Gamma = T_0 ~ \sharp ~\beta$. 

It is not difficult to see that $bi(\Gamma)$ is an invariant of $\Gamma$.

\begin{question}
Are there some connections of $bi(\Gamma)$ with other invariants of $\Gamma$?
\end{question}

\subsection{Groups of spatial graphs} Let $\Gamma$ be a spatial graph, then its group $G_{\Gamma}$ is the fundamental group of the complement $\Gamma$
in 3-space: $G_{\Gamma} = \pi_1(\mathbb{R}^3 - \Gamma)$. We can find the set of generators and defining relations for this group. To do it, consider the regular projection $D_{\Gamma}$ of $\Gamma$  into the plane  and remove all vertices of valence $\geq 3$. Denote  the connected components of this diagram by $x_1, x_2, \ldots, x_s$ this is the generating set of $G_{\Gamma}$. Defining relations can be two types. Any crossing of $D_{\Gamma}$ corresponds defining relation as in Fig. 13.

\begin{figure}[h]
\noindent\centering{\includegraphics[height=0.2 \textwidth]{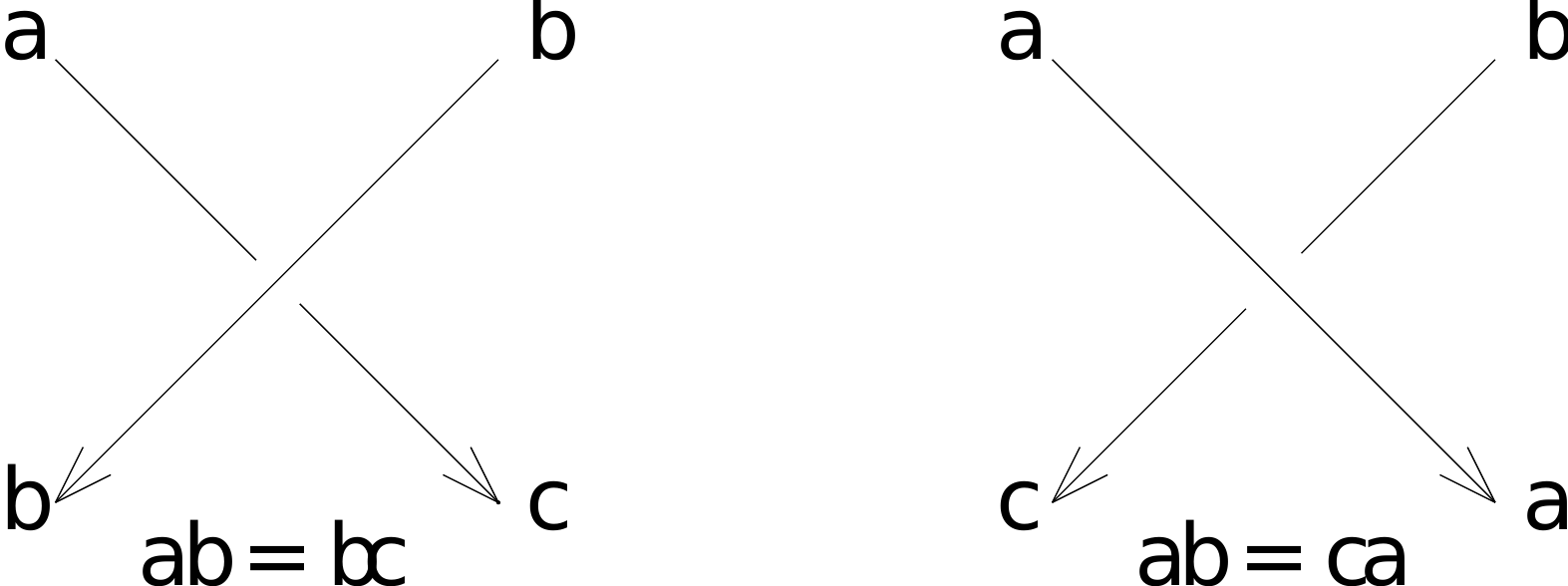}}\\
{\small Fig. 13. Relations in the crossings}
\label{sigma}
\end{figure}

Any vertex of $D_{\Gamma}$ corresponds defining relation as in Figure 14.

\begin{figure}[h]
\noindent\centering{\includegraphics[height=0.3 \textwidth]{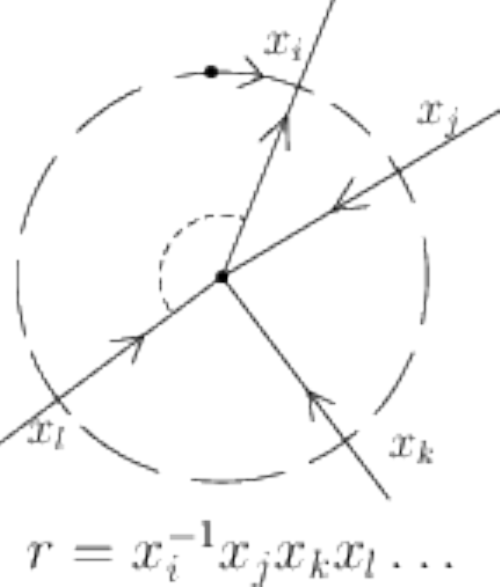}}\\
{\small Fig. 14. Relation in the vertex}
\label{sigma}
\end{figure}

For simplicity we will consider the case when $\Gamma$ is a connected graph.

Suppose we have decomposition $\Gamma = \overline{T}~ \sharp ~ \overline{\beta}$, constructed in Proposition \ref{p1}. To find $G_{\Gamma}$, define a group $G_{\overline{T}}$ and a group $G_{\overline{\beta}}$. Suppose that a 3-ball $B$ contains the forest $\overline{T}$ and $\overline{\beta} = \Gamma \setminus (B \cap \overline{T})$. Then
$$
G_{\overline{T}} = \pi_1 (B \setminus \overline{T}),~~~G_{\overline{\beta}} = \pi_1 \left( (\mathbb{R}^3 \setminus B) \setminus \overline{\beta} \right).
$$
To find $G_{\Gamma}$ note that $G_{\overline{T}} = G_{\overline{T_1}}$, where $\overline{T_1}$ is a tree that is gotten from ${\overline{T}}$ by contracting of the maximal tree $T_0$ into a vertex. Hence, $\overline{T_1}$ contains  one vertex and $m$ incoming edges $a_1, a_2, \ldots, a_m$ and $m$ outgoing edges $b_1, b_2, \ldots, b_m$ and
$$
G_{\overline{T_1}} = \langle a_1, a_2, \ldots, a_m, b_1, b_2, \ldots, b_m ~||~ a_1 a_2 \ldots a_m b_m^{-1} b_{m-2}^{-1} \ldots b_1^{-1} = 1 \rangle.
$$

Let $b'_1, b'_2, \ldots, b'_m$ be the set of incoming strings in $\overline{\beta}$ and $a'_1, a'_2, \ldots, a'_m$ be the set of outcoming strings in $\overline{\beta}$. Then the group $G_{\overline{\beta}}$ contains elements
$$
a'_1, a'_2, \ldots, a'_m, b'_1, b'_2, \ldots, b'_m.
$$
Suppose that $G_{\overline{\beta}}$ is defined by a set of generators $\mathcal{X}$ and the set of relations $\mathcal{R}$, i.e.
$$
G_{\overline{\beta}} = \langle \mathcal{X} ~||~\mathcal{R} \rangle.
$$

Then from Van Kampen theorem follows

\begin{theorem} \label{t2}
The group $G_{\Gamma}$ is generated by elements
$$
\mathcal{X}, a_1, a_2, \ldots, a_m, b_1, b_2, \ldots, b_m
$$
and is defined by the relations
$$
\mathcal{R}, a_1 a_2 \ldots a_m b_m^{-1} b_{m-2}^{-1} \ldots b_1^{-1} = 1,~~b_1 = b'_m, b_2 = b'_{m-1}, \ldots, b_m = b'_1,
$$
$$
a_1 = a'_m, a_2 = a'_{m-1}, \ldots, a_m = a'_1.
$$
\end{theorem}

In particular, if $\Gamma$ is a link then we can decompose it in the form $\Gamma = \overline{T}~ \sharp ~ \overline{\beta}$, where $\overline{T}$ is the disjoint union of $m$ edges $b_1 = a_1$, $b_2 = a_2$, $\ldots$, $b_b = a_m$ and we have

\begin{cor}
The group $G_{\Gamma}$ is generated by elements $\mathcal{X}$ and is defined by the relations
$$
\mathcal{R}, ~~a'_1 = b'_1, a'_2 = b'_{2}, \ldots, a'_m = b'_m.
$$
\end{cor}

\bigskip

In \cite{K1} was considered conception of {\it unknotted spatial graph}. In our terms we can reformulate it by the following manner. A spatial graph $\Gamma$ is called unknotted if there is a decomposition $\Gamma = \overline{T}~ \sharp ~ \overline{\beta}$ into a connection sum of a forest and a tangle, where the tangle $\overline{\beta}$ is monotone. A tangle is called monotone if it has a monotone diagram. A strand in a diagram of a tangle is called {\it monotone} if a point going along the oriented strands meets first the upper crossing point at every crossing point. A diagram of a tangle is called {\it monotone} if every its strand is monotone and  there is some ordering of strands such that the strand with number $i$ is upper than the strands with number $j$ for $i < j$. 

As corollary of Theorem \ref{t2} we get

\begin{cor}
Suppose that a finite connected spatial graph $\Gamma$ has a decomposition $\Gamma = \overline{T}~ \sharp ~ \overline{\beta}$. Then
 
 1) if $\overline{\beta}$ is a $(1,1)$-tangle, then $G_{\Gamma}$ is isomorphic to the group $G_K$ of the knot $K$, which is the closure of the tangle $\overline{\beta}$;
 
2) if $\overline{\beta}$ is a monotone tangle with $k$ strands. Then  $G_{\Gamma}$ is the free  group of rank $k$.
\end{cor}

\begin{proof}

1) In this case $\overline{T}$ is a line segment with a finite set of vertexes. If we compress this segment into a vertex, then our spatial graph becomes a knot $K$ that is the closure of $\overline{\beta}$.

2) Follows from  \cite{K1}, where  was proved the following assertion: An unknotted connected spatial graph is equivalent to a trivial bouquet of circles after the edge contraction of a maximal tree. In particular, the group of unknotted connected spatial graph is free.

\end{proof}

\section{Conclusion remarks}

Of cause, the decomposition of the spatial graph $\Gamma$ in Proposition \ref{p1} is not unique, because  there are different possibilities in the choice of the maximal forest. We can formulate

\begin{conjecture}
Let $\Gamma$ and $\Gamma'$ are two finite spatial graphs, which correspond to some combinatorial graph $\mathcal{G}$. They are equivalent if and only if there are decompositions $\Gamma = \overline{T}~ \sharp ~ \overline{\beta}$ and $\Gamma' = \overline{T'}~ \sharp ~ \overline{\beta'}$ into the connected sum of a forest and a tangle such that the forest $\overline{T}$ is equivalent to the forest $\overline{T'}$ (as spatial graphs) and the tangle $\overline{\beta}$ is equivalent to the tangle $\overline{\beta'}$.
\end{conjecture}

When we have decomposition of two spatial graphs in the connected sum of a forest and a braid, then the equivalence decision problem is more complicated and we formulate

\begin{question}
Let $\Gamma = T_0~ \sharp ~ \beta$ and $\Gamma = T_0'~ \sharp ~ \beta'$ be two decompositions of some finite spatial graph $\Gamma$. How are they related? Is it
possible to prove some analog of Markov theorem for spatial graphs?
\end{question}

\medskip

We know the following problem: for the spatial graph $\Gamma$ find all links, which can be
embedding in $\Gamma$. A {\it constituent link} of a spatial graph $\Gamma$ is a link contained in $\Gamma$. Conway and Gordon \cite{CG} proved that every spatial 6-complete graph $K_6$ contains a non trivial constituent link and every spatial 7-complete graph $K_7$ contains a non trivial constituent knot.
Is it possible to prove these results, using a decomposition of a spatial graph, constructed in the present paper?

\medskip

We know the construction of a knot quandle.

\begin{question}
Is it possible to define for any spatial graph $\Gamma$ a quandle $Q_{\Gamma}$ such that, if $\Gamma'$ is equivalent to $\Gamma$, the $Q_{\Gamma}$ is isomorphic to $Q_{\Gamma'}$?
\end{question}

\medskip


 \end{document}